\documentclass[10pt]{amsart}
\usepackage{graphicx}
\usepackage{latexsym}
\usepackage{fancyhdr}
\usepackage{amsmath, amssymb}
 \usepackage[utf8]{inputenc}
\usepackage[all]{xy}
\usepackage{pdflscape}
\usepackage{longtable}
\usepackage{placeins}
\usepackage{float}
\usepackage{rotating}
\usepackage{tikz}
\usepackage{hyperref}
\usepackage{subfigure}
\usepackage{mathrsfs}

\usepackage{tensor}
\usepackage{enumerate}
\usepackage{mathtools}

\usepackage{color}
\usetikzlibrary{patterns}

\newlength{\hatchspread}
\newlength{\hatchthickness}
\newlength{\hatchshift}
\newcommand{\hatchcolor}{}
\tikzset{hatchspread/.code={\setlength{\hatchspread}{#1}},
         hatchthickness/.code={\setlength{\hatchthickness}{#1}},
         hatchshift/.code={\setlength{\hatchshift}{#1}},
         hatchcolor/.code={\renewcommand{\hatchcolor}{#1}}}
\tikzset{hatchspread=2pt,
         hatchthickness=0.2pt,
         hatchshift=0pt,
         hatchcolor=black}
\pgfdeclarepatternformonly[\hatchspread,\hatchthickness,\hatchshift,\hatchcolor]
   {custom north east lines}
   {\pgfqpoint{\dimexpr 0\hatchthickness}{\dimexpr0\hatchthickness}}
   {\pgfqpoint{\dimexpr\hatchspread+2\hatchthickness}{\dimexpr\hatchspread+2\hatchthickness}}
   {\pgfqpoint{\dimexpr\hatchspread}{\dimexpr\hatchspread}}
   {
    \pgfsetlinewidth{\hatchthickness}
    \pgfpathmoveto{\pgfqpoint{\dimexpr\hatchshift-0.15pt}{-0.15pt}}
    \pgfpathlineto{\pgfqpoint{\dimexpr\hatchspread+0.15pt}{\dimexpr\hatchspread-\hatchshift+0.15pt}}
    \ifdim \hatchshift > 0pt
      \pgfpathmoveto{\pgfqpoint{-0.15pt}{\dimexpr\hatchspread-\hatchshift-0.15pt}}
      \pgfpathlineto{\pgfqpoint{\dimexpr\hatchshift+0.15pt}{\dimexpr\hatchspread+0.15pt}}
    \fi
    \pgfsetstrokecolor{\hatchcolor}
    \pgfusepath{stroke}
   }

\theoremstyle{plain}
\newtheorem{thm}{Theorem}[section]
\newtheorem*{main}{Main~Theorem}
\newtheorem*{theo}{Theorem}
\newtheorem{cor}[thm]{Corollary}
\newtheorem{lem}[thm]{Lemma} 
\newtheorem{prop}[thm]{Proposition}

\theoremstyle{definition}
\newtheorem{defi}[thm]{Definition}
\theoremstyle{remark}
\newtheorem{rem}[thm]{Remark}
\numberwithin{equation}{section}

\newtheorem{ex}[thm]{Example}

\definecolor{aquamarine}{rgb}{0.5, 1.0, 0.83}
\definecolor{bananamania}{rgb}{0.98, 0.91, 0.71}
\definecolor{blizzardblue}{rgb}{0.67, 0.9, 0.93}
\definecolor{corn}{rgb}{0.98, 0.93, 0.36}
\definecolor{lightgreen}{rgb}{0.76, 0.98, 0.76}
\definecolor{lightcoral}{rgb}{0.97, 0.75, 0.75}
\definecolor{lightblue}{rgb}{0.68, 0.85, 0.9}

\newcommand{\lgw}{\longrightarrow}
\newcommand{\lgm}{\longmapsto}

\newcommand{\Ra}{\operatorname{Raf}^-}
\newcommand{\Raf}{\operatorname{Raf}^+}

\newcommand{\G}{\Gamma}

\renewcommand{\O}{\mathcal{O}}

\newcommand{\Aut}{\operatorname{Aut}}

\newcommand{\Z}{\mathbb{Z}}

\renewcommand{\k}{\Bbbk}

\newcommand{\R}{\mathbb{R}}
\newcommand{\U}{\mathcal U}

\newcommand{\LEQ}{\,\boldsymbol{\leq}\,}

\newcommand{\m}{\mathfrak m}
\newcommand{\SEQ}{\,\boldsymbol{<}\,}
\newcommand{\K}{\mathbb{K}}

\newcommand{\N}{\mathbb{N}}

\newcommand{\Q}{\mathbb{Q}}

\newcommand{\rk}{\operatorname{rank}}

\renewcommand{\phi}{\varphi}

\newcommand{\ZR}{\operatorname{ZR}}
\newcommand{\Ord}{\operatorname{Ord}}

\newcommand{\SZR}{\operatorname{SZR}}
\newcommand{\SOrd}{\operatorname{SOrd}}

\begin{document}
\baselineskip=13pt
%
\title{Some algebraic and topological properties of subspaces of (pre)orders}

\author{Julie Decaup}
\email{julie.decaup@im.unam.mx}
\address{Instituto de Matem\'aticas, Universidad Nacional Aut\'onoma de M\'exico (UNAM), Mexico}



\begin{abstract}
We study  algebraic and topological properties of subsets of preorders on a group. In particular we study properties of the composition of two preorders, generalize a topological theorem of \cite{S} in the case of standard orders and show the same theorem in the case of standard preorders. We also show a property of standard valuations.
\end{abstract}
\maketitle
\tableofcontents
\section{Introduction}
Historically, the study of orderable groups has been developed since the  end of the nineteen century because of their importance in algebraic topology. But the first study of the topological properties of the set of orders on a group is due to Kuroda in the case $G=\Z^n$ \cite{Ku}, and to Sikora in the general case \cite{S}. Here, an order means a total order that is left-invariant. In his paper, Sikora introduced a topology on the set of orders on a group, and showed that this topology is a metric topology in the case of countable groups. For a countable group $G$, Sikora proved that the space of left-invariant orders (denoted by $\Ord_l(G)$) on $G$ is a compact metric space, and shows that this is even a Cantor set when $G=\Z^n$. Subsequently, several authors proved that $\Ord_l(G)$ is a Cantor set for several examples of groups $G$.  \\
The first study of the space of preorders on a group $G$ is due to Ewald and Ishida \cite{EI} for $G=\Z^n$. Let us mention that a preorder satisfies all the properties of an order except that it may not be antisymmetric. In their paper, they introduce a topology of the set of preorders on $\Z^n$ (extending the one of Kuroda), and show the compacity of this set.
\\
In \cite{DR}, G. Rond and the author studied some algebraic, combinatorial and topological properties of the space of preorders on a group. They proved that the space of preorders $\ZR_\ast(G)$ is compact for three topologies introduced in the same paper and that we will recall here (see Subsection \ref{sub topologie}). They also introduced the composition of two preorders, the rank of a preorder and the degree of a preorder. All these definitions will be given in the first section (see Section \ref{preliminaires}).\\\\
Here we study some properties of the composition $\preceq_1 \circ \preceq_2$ of two preorders $\preceq_1$ and $\preceq_2$. In particular we show (see Corollary \ref{borne rank composee}) that 
$$\rk(\preceq_1)\leq\rk(\preceq_1 \circ \preceq_2)\leq \rk(\preceq_1)+\rk(\preceq_2).$$
Then we prove the following theorem (see Theorem \ref{decomp Q}):
\begin{theo}
Let $G=\Q^n$ for some $n\in \N$ and $\preceq \in \ZR(G) \setminus \{\preceq_\emptyset\}$.\\ Then there exists $(\preceq_i)_{1\leq i \leq n} \in \ZR(G)$ of rank one such that $\preceq=\preceq_1\circ\cdots\circ\preceq_n$.
\end{theo}

We also show that this result is not true for a general group.\\
In \cite{S}, Sikora introduced the standard orders and showed that, if for all $k$ we have $G_k/G_{k+1}$ finitely generated, with $(G_k)_k$ the lower central series of a countable group $G$, the set of standard orders is either empty or homeomorphic to a Cantor set. Here we generalize his definition to standard preorders and study some topological properties of the set of standard preorders. In particular, we show that the set of standard preorders is compact for the three topologies we consider.\\
The main result of Section $4$, and of this paper, is the following:
\begin{main}
Assume $G \neq \Z$. Then
\begin{enumerate}
\item The set of standard orders is either empty or homeomorphic to the Cantor Set.\\
\item The set of standard preorders is either trivial or homeomorphic to the Cantor Set.
\end{enumerate}
\end{main}
So we were able, as suggested in \cite{S}, to remove the hypotesis on the quotients of the subgroups of the lower central series that gives Sikora in \cite{S} and we generalize this theorem to the preorders (see Theorem \ref{standardordercantor}).\\\\
In the last section, we look at the link between valuations and standard preorders. In \cite{DR}, we proved that a preorder $\preceq$ on a group $G$ defines a monomial valuation $\nu\colon \K^\k_G:=\{\sum\limits_{g\in G} a_g x^g \mid a_g \in \k \} \to \Gamma$ with $\Gamma$ some group depending on $G$ and $\preceq$ and $\k$ some field. Here we study what happens when the preorder is a standard one and we show that, except in a few cases, if we take $P(x) \in \K^\k_G$, the element $x^hP(x)$, where $h$ is some element of $G$, is in the maximal ideal of the valuation ring (that is local) of the valuation $\nu$. Something to remark here is that $h$ does not depend on $P$.

\section{The Zariski-Riemann space of preorders}\label{preliminaires}
This section is made to give the basic definitions and results that will be usefull to read the paper. All the results of this section have been proven in \cite{DR}.
\subsection{Generalities}
\begin{defi}\label{preorder}
Let $G$ be a group. We denote by $\ZR_l(G)$ the set of left-invariant preorders on $G$, i.e. the set of binary relations $\preceq$ on $G$ such that
\begin{itemize}
\item[i)] $\forall u, v\in G$, either $u\preceq v$ or $v\preceq u$,
\item[ii)] $\forall u,v,w\in G$, $u\preceq v$ and $v\preceq w$ implies $u\preceq w$,
\item[iii)] (\emph{left invariance}) $\forall u,v,w\in G$, $u\preceq v$ implies $wu\preceq wv$.
\end{itemize}
In the same way, we define right-invariant preorders whose set is denoted by $\ZR_r(G)$. The set of preorders that are bi-invariant, that is left and right-invariant, is denoted by $\ZR(G)$.
The subset of orders of $\ZR_*(G)$  is denoted by $\Ord_*(G)$ for $*=l$, $r$ or $ \emptyset$.\\
The trivial preorder, i.e. the unique preorder $\preceq$ such that $u\preceq v$ for every $u$, $v\in G$ is denoted by $\leq_\emptyset$.
\end{defi}

\begin{defi}
Let $G$ be a group.
For $\preceq\in\ZR_*(G)$ and $u$, $v\in G$, we write $u\prec v$ if 
$$u\preceq v\text{ and }\neg(v\preceq u).$$
\end{defi}

\begin{defi}\label{def_equiv}
Let $G$ be a group.
Let $\preceq\in \ZR_*(G)$. We define an equivalence relation $\sim_\preceq$ as follows: 
$$u\sim_\preceq v \text{ if } u\preceq v \text{ and } v\preceq u.$$
This equivalence relation is compatible with the group law if $\preceq$ is bi-invariant. In this case the quotient $G/\sim_\preceq$ is a group denoted by $G^\preceq$ and $\preceq$ induces in an obvious way an order on $G^\preceq$ still denoted by $\preceq$.

\end{defi}


\subsection{Ordering of the set of orders}
\begin{defi}
Given two preorders $\preceq_1$, $\preceq_2\in\ZR_\ast(G)$ where $G$ is a group, we say that $\preceq_2$ refines $\preceq_1$ if
$$\forall u,v\in G, u\preceq_2 v\Longrightarrow u\preceq_1 v.$$
\end{defi}

\begin{rem}\label{rem1}
Let $\preceq_1$, $\preceq_2\in\ZR_\ast(G)$. If $\preceq_1$ refines $\preceq_2$ and $\preceq_2$ refines $\preceq_1$ then $\preceq_1=\preceq_2$.
\end{rem}

\begin{rem}\label{contr}
By contraposition, $\preceq_2$ refines $\preceq_1$ if and only if 
$$\forall u,v\in G, u\prec_1 v\Longrightarrow u\prec_2 v.$$
\end{rem}


\begin{defi} Let $G$ be a group.
We define a partial order  $\boldsymbol{\leq}$ on $\ZR_\ast(G)$ as follows: for every preorders $\preceq_1$, $\preceq_2\in\ZR_\ast(G)$ we have
$$\preceq_1\ \boldsymbol{\leq} \ \preceq_2$$
if $\preceq_2$ is a refinement of $\preceq_1$. By Remark \ref{rem1} it is straightforward to check that $\boldsymbol{\leq}$ is a partial order.
\end{defi}

Let $G$ be a group and let $\preceq \in \ZR_*(G)$. We  set 
$$\Ra_*(\preceq):=\{\preceq'\in\ZR_*(G) \text{ such that } \preceq' \LEQ \preceq\},$$
$$\Raf_*(\preceq):=\{\preceq'\in\ZR_*(G) \text{ such that } \preceq \LEQ \preceq'\}.$$

\begin{thm}[\cite{DR}, Theorem 2.19]\label{cor_raf_toset}
Let $G$ be a group. Then $\ZR_*(G)$ is a join-semilattice, that is a partially ordered set in which all subsets have  an infimum.\\
Moreover, for every  $\preceq\in\ZR_*(G)$,
 $(\Ra_*(\preceq), \,\boldsymbol{\leq}) $ is a totally ordered set.
\end{thm}
\subsection{Topologies}\label{sub topologie}

\begin{defi}
Let $G$ be a group. The \emph{Zariski  topology} on $\ZR_*(G)$ (or Z-topology for short) is the topology
 for  which the sets
$$\mathcal O_{u}:=\{\preceq\in\ZR_*(G)\mid u\succeq 1\},$$
where $u$   runs over the elements of $G$, form a basis of open sets.  
\end{defi}

\begin{defi}
Let $G$ be a group.
The set $\ZR_*(G)$ is endowed with a topology for  which the sets
$$\mathcal U_{u}:=\{\preceq\in\ZR_*(G)\mid u\succ 1\},$$
where $u$  runs over the elements of $G$, form a basis of open sets. This topology is called the \emph{Inverse topology} or I-topology. 
\end{defi}

\begin{rem}
The I-topology and the Z-topology agree on $\Ord_*(G)$.
\end{rem}

\begin{defi}
The \emph{Patch topology} on $\ZR_*(G)$ (or P-topology for short) is the topology for which the sets $\U_{u}$ and $\O_{u}$, where $u$  runs over $G$, form a basis of open sets. This is the coarsest topology finer than the Zariski and the Inverse topologies.
\end{defi}

\begin{thm}[\cite{DR}, Theorem 2.30] \label{compactZR}
Let $G$ be a group. Then $\ZR_\ast(G)$ is compact for the Z-topology, the I-topology and the P-topology.
\end{thm}


\subsection{Rank and degree of a preorder}

\begin{defi}
Let $G$ be a group and $\preceq\in \ZR_*(G)$. We denote by $\#\Ra_*(\preceq)$ the cardinal of $\Ra_*(\preceq)$ (as an initial ordinal). We define the rank of $\preceq$ in $\ZR_*(G)$ to be 
$$\rk_*(\preceq):=\left\{\begin{array}{cc}\displaystyle\#\Ra_*(\preceq)-1 & \text{ if this cardinal is finite} \\
\displaystyle\#\Ra_*(\preceq) & \text{ if this cardinal is infinite}  \end{array}\right.$$
The subset of $\ZR_*(G)$ of preorders of rank equal to $r$ (resp. greater or equal to $r$) is denoted by $\ZR^r_*(G)$ (resp. $\ZR^{\geq r}_*(G)$). 

\end{defi}

\begin{defi}
Let $G$ be a group and $\preceq\in \ZR_*(G)$. The degree of $\preceq$ in $\ZR_*(G)$ is 
$$\deg_*(\preceq):=\left\{\begin{array}{cc}\displaystyle\sup_{\preceq'\in\Ord_*(G))\cap\Raf_*(\preceq)}\#\left(\Ra_*(\preceq')\setminus\Ra_*(\preceq)\right)-1 & \text{ if this supremum is finite} \\
\displaystyle\sup_{\preceq'\in\Ord_*(G))\cap\Raf_*(\preceq)}\#\left(\Ra_*(\preceq')\setminus\Ra_*(\preceq)\right) & \text{ if this supremum is infinite}  \end{array}\right.$$
The subset of $\ZR_*(G)$ of preorders of degree  equal to $d$ (resp. less or equal to $d$) is denoted by $^d\!\ZR_*(G)$ (resp. $^{\leq d}\!ZR_*(G)$).

\end{defi}

\subsection{Action on the set of preorders}

\begin{defi}
Let $G$ be a group and let $\Aut(G)$ be the group of automorphisms of $G$. Then there is a left action of $\Aut(G)$  on $\ZR_*(G)$ defined as follows:
$$ (\phi,\preceq)\in \Aut(G)\times\ZR_*(G)\lgm\  \preceq_\phi$$
defined by
$$ \forall u,v\in G,\ u  \preceq_\phi v \text{ if } \phi(u)\preceq \phi(v).$$
\end{defi}

\subsection{Example: description of $\ZR(\Q^n)$} 
We recall the following result of Robbiano:

\begin{thm}\cite[Theorem 4]{R}\label{rob}
Let $\preceq\in\ZR(\Q^n)$. Then there exist an integer $s\geq 0$ and vectors $u_1$,\ldots, $u_s\in\R^n$ such that
$$\forall u,v\in \Q^n,\ u\preceq v \Longleftrightarrow (u\cdot u_1,\ldots, u\cdot u_s)\leq_{\text{lex}} (v\cdot u_1,\ldots, v\cdot u_s).$$
Then we write $\preceq=\leq_{\left(u_1,\dots,u_s\right)}$.
\end{thm}

\begin{prop}
For a given non trivial preorder $\preceq\in\Q^n$, let $s$ be the smallest integer $s$ satisfying Theorem \ref{rob}. Then the rank of $\preceq$ is $s$. 
\end{prop}
Indeed, $\Ra(\preceq)=\{\preceq_\emptyset,\preceq_{u_1},\cdots,\preceq_{u_1,\cdots,u_s}\}$
 
By \cite{DR} (Section 3.2), we can represent $\ZR(\Q^n)$ as a tree. Every preorder corresponds to a vertex of the graph. For a preorder $\preceq\neq \leq_\emptyset$, we consider the largest preorder $\preceq'$ such that $\preceq'\SEQ\preceq$. Every such a pair $(\preceq,\preceq')$ corresponds to an edge between $\preceq$ and $\preceq'$. Moreover $\ZR(\Q^n)$ is a rooted tree by designating $\leq_\emptyset$ to be the root. 
 \begin{ex}\label{ZRZ2}

For $n=1$, $\ZR(\Q)$ consists of three elements: the trivial preorder $\leq_\emptyset$ for which $u\leq_\emptyset v$ for every $u$, $v\in{\R_\geq 0}$, and the orders $\leq_1$ and $\leq_{-1}$. Since $\leq_1$ and $\leq_{-1}$ are the two refinements of $\leq_\emptyset$, $\ZR(\Q)$ is a rooted tree with two vertices:

\begin{figure}[H]\fbox{\begin{tikzpicture}[scale=6.9]
\draw[line width=0.5pt] (0,0) -- (1,0);
\draw[line width=0.5pt] (0,0) -- (-1,0);

\filldraw
(0,0) circle (0.5pt) node[align=left,   below] {$\leq_\emptyset$}
 (1,0) circle (0.5pt) node[align=left,   below] {$\leq_1$}
 (-1,0) circle (0.5pt) node[align=right,   below] {$\leq_{-1}$}; 
 
    \end{tikzpicture}}
\caption{The tree $\ZR(\Q)$}
    \end{figure}

\end{ex}
\FloatBarrier
\begin{ex}\label{ex Z3}
In dimension 3, we have the following picture:
\begin{figure}[H]
\fbox{\begin{tikzpicture}[scale=0.87]

\draw[line width=0.5pt] (0,0) -- (2,0);

\draw[line width=0.5pt] (0,0) -- (-1.6,-.7);
\draw[line width=0.5pt] (0,0) -- (-1.69,-1.1);
\draw[line width=0.5pt] (0,0) -- (-0.6,1.5);
\draw[line width=0.5pt] (0,0) -- (0.5,-1.7);
\draw[line width=0.5pt] (0,0) -- (1.92,0.55);
\draw[line width=0.5pt] (0,0) -- (1.79,0.9);
\draw[line width=0.5pt] (0,0) -- (.2,1.9);
\draw[line width=0.5pt] (0,0) -- (-1.7,0.55);
\draw[line width=0.5pt] (0,0) -- (-0.9,-.2);
\draw[line width=0.5pt] (0,0) -- (-.3,-1.2);
\draw[line width=0.5pt] (0.5,-1.7)-- (0.65,-2.5);
\draw[line width=0.5pt] (0.5,-1.7)-- (0.85,-2.45);
\draw[line width=0.5pt] (-1.69,-1.1) -- (-2.5,-1.3);
\draw[line width=0.5pt] (-1.69,-1.1) -- (-2.15,-1.8);
\draw[line width=0.5pt] (1.4,-0.43) -- (2.3,-0.5);
\draw[line width=0.5pt] (1.4,-0.43) -- (2.2,-0.85);
\draw[line width=0.5pt] (-0.6,1.5) --(-0.75,2.3);
\draw[line width=0.5pt] (-0.6,1.5) -- (-1.15,2.15);

\draw[line width=0.5pt] (0,0) -- (-2,0) ;

\draw[line width=0.5pt] (0,0) -- (0,2) ;

\draw[line width=0.5pt] (0,0) -- (1.4,-0.43);

\draw[line width=0.5pt] (0,0) -- (-0.6,-0.57);
\draw[line width=0.5pt] (0,0) -- (0.6,0.57);
\draw[line width=0.5pt] (0,0) -- (-0.6,0.57);

\draw[line width=0.5pt] (0,0) -- (0,-2) ;

\draw[line width=0.5pt] (0,0) -- (0.7,1.3);

\draw[line width=0.5pt] (2,0) -- (2.77,0.3) ;

\draw[line width=0.5pt] (2,0) -- (3,1) ;
\draw[line width=0.5pt] (2,0) -- (3,-1) ;
\draw[line width=0.5pt] (2,0) -- (3.24,0.2) ;
\draw[line width=0.5pt] (2,0) -- (3.23,-0.4) ;
\draw[line width=0.5pt] (2,0) -- (2.81,-0.6) ;
\draw[line width=0.5pt] (2,0) -- (3.25,-0.15) ;
\draw[line width=0.5pt] (2,0) -- (3.18,0.7) ;

\draw[line width=0.5pt] (3.25,-0.15) -- (4,-0.25) ;
\draw[line width=0.5pt] (3.25,-0.15) -- (4,-0.05) ;

\draw[line width=0.5pt] (2.81,-0.6) -- (3.6,-0.5) ;
\draw[line width=0.5pt] (2.81,-0.6) -- (3.6,-0.7) ;

\draw[line width=0.5pt] (3,1) -- (3.8,1.10) ;
\draw[line width=0.5pt] (3,1) -- (3.8,.90) ;

\draw[line width=0.5pt] (3,-1) -- (3.8,-1.10) ;
\draw[line width=0.5pt] (3,-1) -- (3.8,-.90) ;

\draw[line width=0.5pt] (-2,0) -- (-2.77,0.3) ;

\draw[line width=0.5pt] (-2,0) -- (-3,1) ;
\draw[line width=0.5pt] (-2,0) -- (-3,-1) ;
\draw[line width=0.5pt] (-2,0) -- (-3.24,0.2) ;
\draw[line width=0.5pt] (-2,0) -- (-3.23,-0.4) ;
\draw[line width=0.5pt] (-2,0) -- (-2.81,-0.6) ;
\draw[line width=0.5pt] (-2,0) -- (-3.18,0.7) ;
\draw[line width=0.5pt] (-2,0) -- (-3.25,-0.15) ;

\draw[line width=0.5pt] (-3.25,-0.15) -- (-4,-0.25) ;
\draw[line width=0.5pt] (-3.25,-0.15) -- (-4,-0.05) ;

\draw[line width=0.5pt] (-2.81,-0.6) -- (-3.6,-0.5) ;
\draw[line width=0.5pt] (-2.81,-0.6) -- (-3.6,-0.7) ;

\draw[line width=0.5pt] (-3,1) -- (-3.8,1.10) ;
\draw[line width=0.5pt] (-3,1) -- (-3.8,.90) ;

\draw[line width=0.5pt] (-3,-1) -- (-3.8,-1.10) ;
\draw[line width=0.5pt] (-3,-1) -- (-3.8,-.90) ;
\draw[line width=0.5pt] (0,-2) -- (0.3,-2.77) ;

\draw[line width=0.5pt] (0,-2) -- (1,-3) ;
\draw[line width=0.5pt] (0,-2) -- (-1,-3) ;
\draw[line width=0.5pt] (0,-2) -- (0.2,-3.24) ;
\draw[line width=0.5pt] (0,-2) -- (-0.4,-3.23) ;
\draw[line width=0.5pt] (0,-2) -- (-0.6,-2.81) ;
\draw[line width=0.5pt] (0,-2) -- (-0.15,-3.25) ;
\draw[line width=0.5pt] (0,-2) -- (0.7,-3.18) ;

\draw[line width=0.5pt] (-0.15,-3.25) -- (-0.25,-4) ;
\draw[line width=0.5pt] (-0.15,-3.25) -- (-0.05,-4) ;

\draw[line width=0.5pt] (-0.6,-2.81) -- (-0.5,-3.6) ;
\draw[line width=0.5pt] (-0.6,-2.81) -- (-0.7,-3.6) ;

\draw[line width=0.5pt] (1,-3) -- (1.10,-3.8) ;
\draw[line width=0.5pt] (1,-3) -- (.90,-3.8) ;

\draw[line width=0.5pt] (-1,-3) -- (-1.10,-3.8) ;
\draw[line width=0.5pt] (-1,-3) -- (-.90,-3.8) ;
\draw[dashed, ultra thin] (0,3) circle [x radius=10mm, y radius=0.25cm];
\draw[dashed, ultra thin] (0,-3) circle [x radius=10mm, y radius=0.25cm];
\draw[line width=0.5pt] (0,2) -- (0.3,2.77) ;

\draw[line width=0.5pt] (0,2) -- (1,3) ;
\draw[line width=0.5pt] (0,2) -- (-1,3) ;
\draw[line width=0.5pt] (0,2) -- (0.2,3.24) ;
\draw[line width=0.5pt] (0,2) -- (-0.4,3.23) ;
\draw[line width=0.5pt] (0,2) -- (-0.6,2.81) ;
\draw[line width=0.5pt] (0,2) -- (-0.15,3.25) ;
\draw[line width=0.5pt] (0,2) -- (0.7,3.18) ;

\draw[line width=0.5pt] (-0.15,3.25) -- (-0.25,4) ;
\draw[line width=0.5pt] (-0.15,3.25) -- (-0.05,4) ;

\draw[line width=0.5pt] (-0.6,2.81) -- (-0.5,3.6) ;
\draw[line width=0.5pt] (-0.6,2.81) -- (-0.7,3.6) ;

\draw[line width=0.5pt] (1,3) -- (1.10,3.8) ;
\draw[line width=0.5pt] (1,3) -- (.90,3.8) ;

\draw[line width=0.5pt] (-1,3) -- (-1.10,3.8) ;
\draw[line width=0.5pt] (-1,3) -- (-.90,3.8) ;

\draw[line width=0.5pt] (0,0) -- (-1.05,-1.7) ;
\draw[line width=0.5pt] (0,0) -- (-1.6,1.2) ;
\draw[line width=0.5pt] (0,0) -- (1.2,-1.3) ;

\draw[line width=0.5pt] (1.2,-1.3) -- (2,-1.6) ;
\draw[line width=0.5pt] (1.2,-1.3) -- (1.5,-2.1) ;

  \draw[dashed, line width=0.01pt] (0,0) circle (2cm);
  \draw[dashed, line width=0.01pt] (-2,0) arc (180:360:2 and 0.6);
  \draw[dashed,line width=0.01pt] (2,0) arc (0:180:2 and 0.6);
  \fill[fill=black] (0,0) circle (1pt);
\draw[dashed, ultra thin] (3,0) circle [x radius=.25cm, y radius=10mm];

\draw[dashed, ultra thin] (-3,0) circle [x radius=.25cm, y radius=10mm];

\filldraw (2,-1.6) circle (1.pt);
\filldraw (1.5,-2.1) circle (1.pt);
\filldraw (-1.05,-1.7) circle (1.pt);
\filldraw (-1.6,1.2) circle (1.pt);
\filldraw (1.2,-1.3) circle (1.pt);
\filldraw (0,0) circle (1.pt);
\filldraw (2,0) circle (1.pt);
\filldraw (-2,0) circle (1.pt);
\filldraw (3,1) circle (1.pt);
\filldraw (3,-1) circle (1.pt);
\filldraw (2.81,-0.6) circle (1.pt);
\filldraw (3.25,-0.15) circle (1.pt);
\filldraw (3.8,-.90) circle (1.pt);

\filldraw (3.8,-1.10) circle (1.pt);
\filldraw (3.6,-0.7) circle (1.pt);
\filldraw (3.6,-0.5) circle (1.pt);
\filldraw (4,-0.05) circle (1.pt);
\filldraw (4,-0.25) circle (1.pt);
\filldraw (3.18,0.7) circle (1.pt);
\filldraw (3.23,-0.4) circle (1.pt);
\filldraw (3.24,0.2) circle (1.pt);
\filldraw (2.77,0.3) circle (1.pt);

\filldraw (0.65,-2.5) circle (1.pt);
\filldraw (0.85,-2.45) circle (1.pt);
\filldraw (-2.5,-1.3) circle (1.pt);
\filldraw  (-2.15,-1.8) circle (1.pt);
\filldraw  (2.3,-0.5) circle (1.pt);
\filldraw  (2.2,-0.85) circle (1.pt);
\filldraw (-0.75,2.3) circle (1.pt);
\filldraw  (-1.15,2.15) circle (1.pt);
\filldraw (-1.6,-.7) circle (1.pt);
\filldraw (-1.69,-1.1) circle (1.pt);
\filldraw (-0.6,1.5) circle (1.pt);
\filldraw (0.5,-1.7) circle (1.pt);
\filldraw (1.92,0.55) circle (1.pt);
\filldraw (.2,1.9) circle (1.pt);
\filldraw (1.79,0.9) circle (1.pt);
\filldraw (-1.7,0.55) circle (1.pt);
\filldraw (-0.9,-.2) circle (1.pt);
\filldraw (-.3,-1.2) circle (1.pt);
\filldraw (1,3) circle (1.pt);
\filldraw (-1,3) circle (1.pt);
\filldraw (-0.6,2.81) circle (1.pt);
\filldraw (-0.15,3.25) circle (1.pt);
\filldraw (-.90,3.8) circle (1.pt);

\filldraw (-1.10,3.8) circle (1.pt);
\filldraw (-0.7,3.6) circle (1.pt);
\filldraw (-0.5,3.6) circle (1.pt);
\filldraw (-0.05,4) circle (1.pt);
\filldraw (-0.25,4) circle (1.pt);
\filldraw (0.7,3.18) circle (1.pt);
\filldraw (-.4,3.23) circle (1.pt);
\filldraw (.2,3.24) circle (1.pt);
\filldraw (.3,2.77) circle (1.pt);

\filldraw (1.1,3.8) circle (1.pt);
\filldraw (.9,3.8) circle (1.pt);
\filldraw (1,-3) circle (1.pt);
\filldraw (-1,-3) circle (1.pt);
\filldraw (-0.6,-2.81) circle (1.pt);
\filldraw (-0.15,-3.25) circle (1.pt);
\filldraw (-.90,-3.8) circle (1.pt);

\filldraw (-1.10,-3.8) circle (1.pt);
\filldraw (-0.7,-3.6) circle (1.pt);
\filldraw (-0.5,-3.6) circle (1.pt);
\filldraw (-0.05,-4) circle (1.pt);
\filldraw (-0.25,-4) circle (1.pt);
\filldraw (0.7,-3.18) circle (1.pt);
\filldraw (-.4,-3.23) circle (1.pt);
\filldraw (.2,-3.24) circle (1.pt);
\filldraw (.3,-2.77) circle (1.pt);

\filldraw (1.1,-3.8) circle (1.pt);
\filldraw (.9,-3.8) circle (1.pt);
\filldraw (-2.81,-0.6) circle (1.pt);
\filldraw (-3.25,-0.15) circle (1.pt);
\filldraw (-3.8,-.90) circle (1.pt);

\filldraw (-3.8,-1.10) circle (1.pt);
\filldraw (-3.6,-0.7) circle (1.pt);
\filldraw (-3.6,-0.5) circle (1.pt);
\filldraw (-4,-0.05) circle (1.pt);
\filldraw (-4,-0.25) circle (1.pt);
\filldraw (-3.18,0.7) circle (1.pt);
\filldraw (-3.23,-0.4) circle (1.pt);
\filldraw (-3.24,0.2) circle (1.pt);
\filldraw (-2.77,0.3) circle (1.pt);

\filldraw (-3.8,1.10) circle (1.pt);
\filldraw (-3.8,.90) circle (1.pt);

\filldraw (-0.6,-0.57) circle (1.pt);

\filldraw (3.8,1.10) circle (1.pt);
\filldraw (3.8,.90) circle (1.pt);
\filldraw (0,2) circle (1.pt);
\filldraw (0,-2) circle (1.pt);
\filldraw (0.6,0.57) circle (1.pt);
\filldraw (-0.6,0.57) circle (1.pt);
\filldraw (1.4,-0.43) circle (1.pt);
\filldraw (0.7,1.3) circle (1.pt);
\end{tikzpicture}}

\caption{The tree $\ZR(\Q^3)$}.
\end{figure}
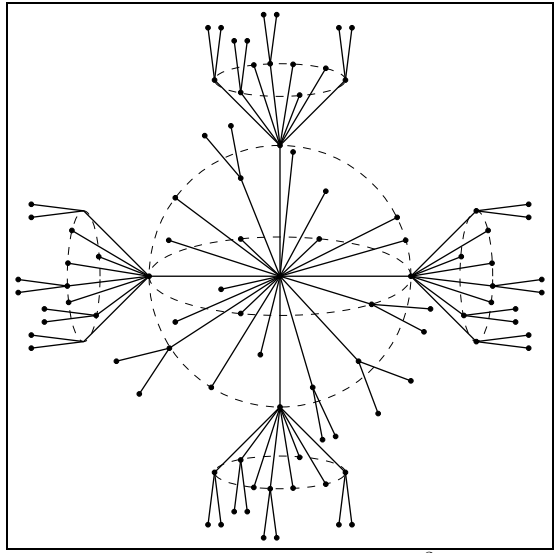

\end{ex}

\FloatBarrier

\section{Composition of preorders}

\begin{defi}Let $G$ be a group and $\preceq_1$,$\preceq_2$ be two elements of $\ZR_\ast(G)$. We define the composition $\preceq=\preceq_1 \circ \preceq_2$ of $\preceq_2$ by $\preceq_1$ as follows:
$$\forall u,v\in G,\ u\preceq v\text{ if }\left\{\begin{array}{c} u\prec_1 v\\
\text{ or } u\sim_{\preceq_1} v \text{ and } u\preceq_2 v.\end{array}\right.$$
\end{defi}

\begin{rem}
In \cite{DR} (Definition 2.48), we gave a more general version of the composition of two preorders, but we will only use this one in this paper.
\end{rem}

\begin{rem}
Let $G$ be a group and $\preceq_1$,$\preceq_2$ be two elements of $\ZR_\ast(G)$. Then $\preceq:=\preceq_1 \circ \preceq_2 \in \ZR_\ast(G)$. It is straightforward to check that $\preceq$ is total and the left and right invariance.\\
Let us show that $u \preceq v$ and $v \preceq w$ implies $u \preceq w$. Assume $u \preceq v$ and $v \preceq w$ and, aiming for contradiction, that $w \prec_1 u$ or $\left( u\sim_{\preceq_1}w \text{ and } w \prec_2 u \right)$. Since $u\preceq v$ and $v \preceq w$ we know that $u \preceq_1 v$ and $v \preceq_1 w$, so $u \preceq_1 w$. Then, we have $u\sim_{\preceq_1}w$ and $w \prec_2 u$. Hence we cannot have $u \prec_1 v$ neither $v\prec_1 w$ so $u\sim_{\preceq_1} v \sim_{\preceq_1} w$ and $u\preceq_2 v \preceq_2 w$. This is a contradiction with the assumption that $w \prec_2 u$.
\end{rem}

\begin{prop}
Let G be a group. The set $\left(\ZR_\ast(G),\circ\right)$ is a monoid, and the identity element is the trivial preorder.
\end{prop}

\begin{proof}
It is straightforward to check that for all $\preceq \in \ZR_\ast(G)$, we have $\preceq \circ \preceq_{\emptyset}=\preceq_{\emptyset} \circ \prec = \preceq$. So we just have to check the associativity property. Let $u,v \in G$ and $\preceq_1, \preceq_2, \preceq_3 \in \ZR_\ast(G)$ and assume $u\preceq_1 \circ \left(\preceq_2 \circ \preceq_3\right)v$, we want to show that $u \left(\preceq_1 \circ \preceq_2 \right) \circ \preceq_3 v$ (the other side is the same proof). Since $u\preceq_1 \circ \left(\preceq_2 \circ \preceq_3\right)v$ we know that $u \prec_1 v$ or $u\sim_{\preceq_1}v$ and $u\preceq_{23}v$, where $\preceq_{ij}:=\preceq_i \circ \preceq_j$.\\
Assume $u \prec_1 v$, then we have $u\prec_{12}v$ and so $u \preceq_{12}\circ\preceq_3 v$. Otherwise we have $u\sim_{\preceq_1} v$ and $u\preceq_{23}v$. So either $u\sim_{\preceq_1} v$ and $u \prec_2 v$, ie $u\prec_{12}v$ hence $u \preceq_{12}\circ\preceq_3 v$, or $u\sim_{\preceq_1} v$ and $u\sim_{\preceq_2} v$ and $u\preceq_3 v$, ie $u\sim_{\preceq_{12}} v$ and $u\preceq_3 v$ and so $u \preceq_{12}\circ\preceq_3 v$. This completes the proof.
\end{proof}

\begin{rem}
The trivial preorder is the only preorder that has an inverse. Indeed, for a given $\preceq \in \ZR_\ast(G)$, if $u,v \in G$ are such that $v \prec_1 u$, we cannot find a $\preceq'$ such that $u \preceq \circ \preceq' v$.\\
We will always denote by $\preceq_{ij}$ the composition $\preceq_i \circ \preceq_j$.
\end{rem}

\begin{ex} Let $G=\Q^2$.\\
\begin{enumerate}
\item $\preceq_{\begin{psmallmatrix} 0 \\  1 \end{psmallmatrix}}\circ \preceq_{\begin{psmallmatrix} 1 \\  0 \end{psmallmatrix}}=\preceq_{\begin{psmallmatrix} 0 \\  1 \end{psmallmatrix},\begin{psmallmatrix} 1 \\  0 \end{psmallmatrix}}$\\
Let us give some details, we take $u=(u_1,u_2)$ and $v=(v_1,v_2)$: 
$$u \preceq_{\begin{psmallmatrix} 0 \\  1 \end{psmallmatrix}}\circ \preceq_{\begin{psmallmatrix} 1 \\  0 \end{psmallmatrix}} v \Leftrightarrow \left\{\begin{array}{c} u\prec_{\begin{psmallmatrix} 0 \\  1 \end{psmallmatrix}} v\\
\text{ or } u\sim_{\preceq_{\begin{psmallmatrix} 0 \\  1 \end{psmallmatrix}}} v \text{ and } u\preceq_{\begin{psmallmatrix} 1 \\  0 \end{psmallmatrix}} v.\end{array}\right. $$
So $u \preceq_{\begin{psmallmatrix} 0 \\  1 \end{psmallmatrix}}\circ \preceq_{\begin{psmallmatrix} 1 \\  0 \end{psmallmatrix}} v \Leftrightarrow \left\{\begin{array}{c} u_2<v_2\\
\text{ or } u_2=v_2 \text{ and } u_1 \leq v_1\end{array}\right. \Leftrightarrow u\preceq_{\begin{psmallmatrix} 0 \\  1 \end{psmallmatrix},\begin{psmallmatrix} 1 \\  0 \end{psmallmatrix}}v.$
\item $\preceq_{\begin{psmallmatrix} 0 \\  1 \end{psmallmatrix}}\circ \preceq_{\begin{psmallmatrix} 1 \\  \sqrt{2} \end{psmallmatrix}}=\preceq_{\begin{psmallmatrix} 0 \\  1 \end{psmallmatrix},\begin{psmallmatrix} 1 \\  0 \end{psmallmatrix}}$\\
Let us give some details, we take $u=(u_1,u_2)$ and $v=(v_1,v_2)$: 
$$u \preceq_{\begin{psmallmatrix} 0 \\  1 \end{psmallmatrix}}\circ \preceq_{\begin{psmallmatrix} 1 \\  \sqrt{2} \end{psmallmatrix}} v \Leftrightarrow \left\{\begin{array}{c} u\prec_{\begin{psmallmatrix} 0 \\  1 \end{psmallmatrix}} v\\
\text{ or } u\sim_{\preceq_{\begin{psmallmatrix} 0 \\  1 \end{psmallmatrix}}} v \text{ and } u\preceq_{\begin{psmallmatrix} 1 \\  \sqrt{2} \end{psmallmatrix}} v.\end{array}\right. $$
So $u \preceq_{\begin{psmallmatrix} 0 \\  1 \end{psmallmatrix}}\circ \preceq_{\begin{psmallmatrix} 1 \\  0 \end{psmallmatrix}} v \Leftrightarrow \left\{\begin{array}{c} u_2<v_2\\
\text{ or } u_2=v_2 \text{ and } u_1+\sqrt{2}u_2 \leq v_1+\sqrt{2}v_2 \end{array}\right. \Leftrightarrow u\preceq_{\begin{psmallmatrix} 0 \\  1 \end{psmallmatrix},\begin{psmallmatrix} 1 \\  0 \end{psmallmatrix}}v$
\item $\preceq_{\begin{psmallmatrix} 1 \\  0 \end{psmallmatrix}}\circ \preceq_{\begin{psmallmatrix} 1 \\  \sqrt{2} \end{psmallmatrix}}=\preceq_{\begin{psmallmatrix} 1 \\  0 \end{psmallmatrix},\begin{psmallmatrix} 0 \\  1 \end{psmallmatrix}}$
\item $\preceq_{\begin{psmallmatrix} 1 \\  1 \end{psmallmatrix}}\circ \preceq_{\begin{psmallmatrix} 1 \\  \sqrt{2} \end{psmallmatrix}}=\preceq_{\begin{psmallmatrix} 1 \\  1 \end{psmallmatrix},\begin{psmallmatrix} 0 \\  1 \end{psmallmatrix}}$
\item $\preceq_{\begin{psmallmatrix} 1 \\  \sqrt{2} \end{psmallmatrix}}\circ \preceq_{\begin{psmallmatrix} a \\  b \end{psmallmatrix}}=\preceq_{\begin{psmallmatrix} 1 \\  \sqrt{2} \end{psmallmatrix}}$
\end{enumerate}
\end{ex}

\begin{rem} \label{ordreetpluspetit}
For all $\preceq_1,\preceq_2 \in \ZR_\ast(G)$, we have $\preceq_1 \boldsymbol{\leq} \preceq_{12}$.\\
If $\preceq_1 \in \Ord_\ast(G)$, then $\preceq_1 \circ \preceq_2 =\preceq_1$ for all $\preceq_2 \in \ZR_\ast(G)$.
\end{rem}

\begin{prop}
Let $\varphi \colon G_1 \to G_2$ a group homomorphism. We consider the application $$\tilde{\varphi} \colon \ZR_\ast(G_2) \to \ZR_\ast(G_1)$$ such that $\tilde{\varphi}(\preceq):= \preceq_\varphi$.\\
Then this application is a monoid homomorphism.
\end{prop}

\begin{proof}
It is clear that $\tilde{\varphi}$ sends the identity element of $\ZR_\ast(G_2)$ onto the identity element of $\ZR_\ast(G_1)$, so we just have to check that $$\tilde{\varphi}\left(\preceq_{12}\right)=\tilde{\varphi}\left(\preceq_1\right)\circ \tilde{\varphi}\left(\preceq_2\right).$$
We have $\tilde{\varphi}\left(\preceq_{12}\right)=\preceq_{12_\varphi}$. So $u \tilde{\varphi}\left(\preceq_{12}\right) v$ if $\varphi(u) \preceq_{12} \varphi(v)$, it means that $\varphi(u) \prec_1 \varphi(v)$ or $\varphi(u) \sim_{\preceq_1} \varphi(v)$ and $\varphi(u) \preceq_2 \varphi(v)$. Hence $u\prec_{1_\varphi}v$ or $u \sim_{\preceq_{1_\varphi}} v$ and $u\preceq_{2_\varphi}v$, so we are done.
\end{proof}

\begin{prop}
Let $\preceq_1,\preceq_2 \in \ZR_\ast(G)$. Then $$G_{\preceq_{12}}=G_{\preceq_{21}}=G_{\preceq_1} \cap G_{\preceq_2}.$$
\end{prop}
\begin{proof}
Let $u$ be an element of $G_{\preceq_{12}}$. Then we have:
\begin{enumerate}
\item $u \prec_1 1$ or $\left(u\in G_{\preceq_1} \text{ and } u\preceq_2 1\right)$
\item \text{ and } $1 \prec_1 u$ or $\left(u\in G_{\preceq_1} \text{ and } 1\preceq_2 u\right)$
\end{enumerate}
So $\left(u\in G_{\preceq_1} \text{ and } u\preceq_2 1\right)$ and $\left(u\in G_{\preceq_1} \text{ and } 1\preceq_2 u\right)$, it means that $u \in G_{\preceq_1} \cap G_{\preceq_2}$ and we have the inclusion $G_{\preceq_{12}} \subseteq G_{\preceq_1} \cap G_{\preceq_2}$. The other one is straightforward to chek, so this concludes the proof.
\end{proof}

\begin{prop} \label{Ra=plusgrand}
Let $\preceq_1 \boldsymbol{\leq} \preceq_2 \in \ZR_\ast(G)$. Then we have $\preceq_{12}=\preceq_{21}=\preceq_2$.
\end{prop}

\begin{proof}
First let us show that $\preceq_{12}=\preceq_2$. Let $u,v$ be two elements of $G$ such that $u \preceq_2 v$. Since $\preceq_1 \boldsymbol{\leq} \preceq_2$, we have $u \preceq_1 v$. So either $u \prec_1 v$ or $u \sim_{\preceq_1} v$ and $u \preceq_2 v$, so we do have $u \preceq_{12} v$.\\
Now let $u,v$ be two elements of $G$ such that $u \preceq_{12} v$, then $u \prec_1 v$ or $u \sim_{\preceq_1}v$ and $u \preceq_2 v$. If $u \prec_1 v$, then $u \preceq_2 v$ since $\preceq_1 \boldsymbol{\leq} \preceq_2$. So assume $u \sim_{\preceq_1}v$ and $u \preceq_2 v$, then in particular $u \preceq_2 v$ and this concludes the proof.\\
Now let us show that $\preceq_{21}=\preceq_2$. We already know that $\preceq_2 \boldsymbol{\leq} \preceq_{21}$, so we just have to prove that $\preceq_{21} \boldsymbol{\leq} \preceq_2$. Let $u,v$ be two elements of $G$ such that $u \preceq_2 v$ and let us prove that $u \preceq_{21} v$. If $u \prec_2 v$, then it is done, otherwise, $u \sim_{\preceq_2} v$ and since $\preceq_1 \boldsymbol{\leq} \preceq_2$, we have $u \sim_{\preceq_1} v$ hence $u \sim_{\preceq_{21}} v$. This concludes the proof.
\end{proof}

\begin{prop}\label{rk minore} Let $\preceq_1,\preceq_2 \in \ZR_\ast(G)$. Then we have:
\begin{enumerate}
\item $\Ra(\preceq_1) \subseteq \Ra(\preceq_{12}).$
\item If $\preceq_1 \boldsymbol{\leq} \preceq_2$, then $\Ra(\preceq_{12})=\Ra(\preceq_{21})=\Ra(\preceq_2).$
\end{enumerate}
\end{prop}

\begin{proof}
Since $\preceq_1 \boldsymbol{\leq} \preceq_{12}$, we have $(1)$. The $(2)$ is given by Proposition \ref{Ra=plusgrand}.
\end{proof}

\begin{rem}
If $\preceq_1 \in \Ord_\ast(G)$ then $\preceq_{12}=\preceq_1$ so the inclusion is optimal. And this is not an equality in general, because, for example, $\preceq_{\begin{psmallmatrix} 0 \\  1 \end{psmallmatrix}}\circ \preceq_{\begin{psmallmatrix} 1 \\  0 \end{psmallmatrix}}=\preceq_{\begin{psmallmatrix} 0 \\  1 \end{psmallmatrix},\begin{psmallmatrix} 1 \\  0 \end{psmallmatrix}}$.
\end{rem}

\begin{prop} \label{surjective} Let $\preceq_1,\preceq_2 \in \ZR_\ast(G)$ and let $f \colon \Ra(\preceq_1) \cup \Ra(\preceq_2)\setminus \{\preceq_{\emptyset}\} \to \Ra(\preceq_{12})$ be such that $f(\preceq)= \preceq$ if $\preceq \in \Ra(\preceq_1)$ and $f(\preceq)=\preceq_1 \circ \preceq$ otherwise.\\
Then $f$ is a surjective map.
\end{prop}

\begin{proof}
It is straightforward to check that this map is well defined. Let us prove that this is a surjective map and let $\preceq \in \Ra(\preceq_{12})\setminus \Ra(\preceq_1)$. We want to show that there exists $\preceq'\in \Ra(\preceq_2)$ such that $\preceq=\preceq_1 \circ \preceq'$.\\
We define $\preceq'$ as follows:
$u\preceq'v$ if $$\left\{  \begin{array}{lllll}
        u \sim_{\preceq}v \\
        \text{or}\\
        u \not\sim_{\preceq_1} v \\
        \text{or}\\
        u \prec v \text{ and } u\prec_2 v
    \end{array} \right.$$ 
It is not hard to see that $\preceq'$ is a preorder. Now, we show that $\preceq' \in \Ra(\preceq_2)$. Let $u,v\in G$ such that $u \preceq_2 v$. We can assume $u\not\sim_{\preceq}v$ and $u\sim_{\preceq_1}v$, hence we have $u\prec v$ and $u\sim_{\preceq_1}v$ or $v \prec u$ and $u\sim_{\preceq_1}v$.
If we have $u\prec v$ and $u\sim_{\preceq_1}v$ then $u\prec_{12} v$ and $u\sim_{\preceq_1}v$ because $\preceq \boldsymbol{\leq}\preceq_{12}$. So we have $u\prec v$ and $u \prec_2 v$, hence $u\preceq'v$. Otherwise we have $v \prec u$ and $u\sim_{\preceq_1}v$ and with the same argument we have $v \prec_2 u$ that contradicts the fact that $u\preceq_2 v$.\\
Now let us prove that $\preceq=\preceq_1 \circ \preceq'$. First we show that $\preceq \boldsymbol{\leq} \preceq_1 \circ \preceq'$. Let $u,v$ such that $u \preceq_1 \circ \preceq' v$, then either $u \prec_1 v$ or $u \sim_{\preceq_1} v$ and $u \preceq'v$. If $u \prec_1 v$, we do have $u \prec v$ since $\preceq_1 \boldsymbol{\leq} \preceq$. Otherwise we have $u \sim_{\preceq_1} v$ and $u \preceq'v$, that is $u \sim_{\preceq} v$ or $u \prec v$ and $u \prec_2 v$. In each case, we do have $u \preceq v$.\\
So let us prove that $\preceq_1 \circ \preceq'\boldsymbol{\leq}\preceq $ and let $u,v$ such that $u \preceq v$. Since $\preceq_1 \boldsymbol{\leq} \preceq$, we have $u \preceq_1 v$, so either $u \prec_1 v$ or $u \sim_{\preceq_1} v$. If $u \prec_1 v$ then $u \preceq_1 \circ \preceq' v$ by definition of the composition, hence we assume $u \sim_{\preceq_1} v$. No we can assume $u \not\sim_{\preceq} v$ otherwise it is trivial, so we have $u \sim_{\preceq_1} v$ and $u \prec v$. Since $\preceq \boldsymbol{\leq} \preceq_{12}$, we have $u \sim_{\preceq_1} v$ and $u \prec_{12} v$, this means we have $u \sim_{\preceq_1} v$ and $u \prec v$ and $u\prec_2 v$, hence $u \preceq_1 \circ \preceq' v$. This completes the proof.
\end{proof}

\begin{rem}
This map is not injective in general. Indeed, take $G=\Q^2$, $\preceq_1=\preceq_{\begin{psmallmatrix} 0 \\  1 \end{psmallmatrix}}$ and $\preceq_2=\preceq_{\begin{psmallmatrix} 0 \\  -1 \end{psmallmatrix}}$. Then $\preceq_2\in \Ra(\preceq_2)\setminus \{\preceq_\emptyset\}$, $\preceq_2 \notin \Ra(\preceq_1)$ but $\preceq_1 \circ \preceq_2=\preceq_1$.
\end{rem}

\begin{cor} \label{borne rank composee}
Let $\preceq_1 , \preceq_2 \in \ZR_\ast(G)$. Then $\rk(\preceq_1)\leq \rk(\preceq_{12})\leq \rk(\preceq_1)+\rk(\preceq_2)$.
\end{cor}

\begin{proof}
This is a consequence of Proposition \ref{rk minore} and Proposition \ref{surjective}.
\end{proof}

\begin{rem}
If $\preceq_{12}=\preceq_1$, for example if $\preceq_1$ is an order, then the first inequality of Corollary \ref{borne rank composee} is in fact an equality and the second one is a strict inequality.\\
Now let us consider $G=\Q^5$ and $\left(e_i\right)_{1\leq i \leq 5}$ the canonical basis of $\R^5$. Let $\preceq_1=\preceq_{e_1,e_2}$, $\rk(\preceq_1)=2$, and $\preceq_2=\preceq_{e_3,e_4,e_5}$, $\rk(\preceq_2)=3$. Then $\preceq_1 \circ \preceq_2=\preceq_{e_1,\dots,e_5}$ is of rank $5$ and the second inequality of Corollary \ref{borne rank composee} is in fact an equality and the first one is a strict inequality. This shows that these inequalities are sharp.
\end{rem}
\subsection{Decomposition of preorders}
One can ask if every preorder of a group can be written as a composition of "easy" preorders: preorders of rank $1$. It means, if $G$ is a group and $\preceq \in \ZR_\ast(G) \setminus \{\preceq_\emptyset\}$, does there exist $(\preceq_i)_{1\leq i \leq n} \in \ZR_\ast^1(G)$ such that $\preceq=\preceq_1\circ\cdots\circ\preceq_n$?\\
The answer is no. In general this is not true, but in some cases it is.

In \cite{DR} (Corollary 2.51) we proved the following:
\begin{thm}
Let $G$ be a group and let $\preceq'$, $\preceq\in\ZR_\ast(G)$.  Then
$$\preceq'\LEQ\preceq\  \Longleftrightarrow \ \exists \preceq_2\in \ZR_\ast(G_{\preceq'}),\ \ \preceq=\preceq'\circ\preceq_2.$$
\end{thm}

In fact we want to decompose a preorder $\preceq$ in a composition $\preceq_1 \circ \preceq_2$ where $\preceq_1$ is the only preorder of rank $1$ such that $\preceq_1 \boldsymbol{\leq} \preceq$ and $\preceq_2$ is of rank strictly inferior than the rank of $\preceq$ and conclude by induction. So we want equality in the last inequality of Corollary \ref{borne rank composee}: $\rk(\preceq_{12})\leq \rk(\preceq_1)+\rk(\preceq_2)$, that means we want that the rank of $\preceq_2$ is exactly $\rk(\preceq)-1$. We obviously have $\preceq=\preceq_1 \circ \preceq$ so an idea is to "move" the segment $\left[\preceq_1, \preceq \right]$( that is totally ordered by Theorem \ref{cor_raf_toset}) to the trivial preorder, or more easily to see, to the root of the tree when the set $\ZR_\ast(G)$ can be seen as a rooted tree (see example \ref{ex Z3}).

For the moment, we prove the result in the case $G=\Q^n$, for all $n \in \N$.

\begin{thm}\label{decomp Q}
Let $G=\Q^n$, with $n \in \N$ and $\preceq \in \ZR_\ast(G) \setminus \{\preceq_\emptyset\}$. Then there exist $(\preceq_i)_{1\leq i \leq n} \in \ZR_\ast^1(G)$ such that $\preceq=\preceq_1\circ\cdots\circ\preceq_n$.
\end{thm}

\begin{proof}
Let $G=\Q^n$ and let $\preceq \in \ZR(\Q^n)$. Then there exist vectors $u_1,\dots,u_s \in \R^n$ such that $\preceq=\preceq_{(u_1,\dots,u_s)}$.\\
We claim that $\preceq_{(u_1,\dots,u_s)}=\preceq_{u_1}\circ \dots \circ \preceq_{u_s}$.\\
Indeed, if $a \preceq_{(u_1,\dots,u_s)} b$ then $$(a.u_1,\dots,a.u_s)\leq_{\text{lex}}(b.u_1,\dots,b.u_s).$$
So $a.u_1<b.u_1$, that is $a\prec_1b$, or $a.u_1=b.u_1$ and $(a.u_2,\dots,a.u_s)\leq_{\text{lex}}(b.u_2,\dots,b.u_s)$, that is $a\sim_{\preceq_{u_1}}b$ or $a\preceq_{(u_2,\dots,u_s)}b$. So $\preceq_{(u_1,\dots,u_s)}=\preceq_{u_1} \circ \preceq_{(u_2,\dots,u_s)}$ and we conclude by induction.
\end{proof}
\begin{rem}
The result does not hold in general. Take for example $$G=\Z^{\infty}:=\{(a_n)_{n \in \N} \text{, }a_n \in \Z \text{ such that } a_n=0 \text{ for }n \text{ large enough}\}.$$
We just have to found an element of $\ZR(G)$ that cannot be decomposed as a composition of rank $1$ preorders.\\
We define $\preceq \in \ZR(G)$ as following: take $(a_n)_{n \in \N}\in G$ and let $k$ be the maximum of the integers such that $a_k \neq 0$. Then we say that $0 \preceq (a_n)_{n \in \N}$ if $a_k \in \N$. This defines a preorder of $G$ that cannot be decomposed as a composition of rank $1$ preorders.
\end{rem}
\section{Standard preorders}
Along this section, $G$ is assumed to be countable and we only consider bi-invariant (pre)orders.\\
Let $G=G_0 \trianglerighteq G_1 \trianglerighteq \cdots \trianglerighteq G_n \trianglerighteq \cdots$ be the lower central series of a group $G$ such that $\bigcap\limits_{k=0}^\infty G_k=\{1\}$.

The following definition was introduced by Sikora in \cite{S} for orders.

\begin{defi}Any preorder $\preceq \in \ZR(G)$ is called a \emph{standard preorder} if for every $g\in G$, if $g\in G_k \setminus G_{k+1}$ and $g$ positive then all elements of $gG_{k+1}$ are positive. We denote by $\SZR(G)$ the set of standard preorders of $G$ and $\SOrd(G)$ the set of standard orders of $G$.
\end{defi}

\begin{prop}\label{standard compo}
Let $\preceq_1, \preceq_2 \in \SZR(G)$, then $\preceq_{12} \in \SZR(G)$.
\end{prop}

\begin{proof}
Let $g \in G$, $g\in G_k\setminus G_{k+1}$ such that $g \succ_{12} 1$, $h\in G_{k+1}$ and let us show that $gh\succ_{12}1$. We know that $g \succ_{12} 1$ so we have either $$g \succ_1 1$$ or $$g \sim_{\preceq_1}1 \text{ and } g \succ_2 1.$$
If $g \succ_1 1$ since $\preceq_1\in\SZR(G)$, we do have $gh\succ_{1}1$ and so $gh\succ_{12}1$. Otherwise we assume $g \sim_{\preceq_1}1 \text{ and } g \succ_2 1.$ Since $\preceq_2\in\SZR(G)$ we have $g \sim_{\preceq_1}1 \text{ and } gh \succ_2 1$. Then $gh \sim_{\preceq_1}h \text{ and } gh \succ_2 1$, so if $h \sim_{\preceq_1} 1$ we do have $gh\succ_{12}1$. Otherwise either $h \succ_1 1$ or $1 \succ_1 h$. If $h \succ_1 1$, since $gh \sim_{\preceq_1}h$ we have $gh \succ_1 1$ and so $gh \succ_{12}1$. Then we assume $1 \succ_1 h$, ie $g^{-1}h^{-1} \succ_1 1$ because $g \sim_{\preceq_1} 1$. Since $\preceq_1\in\SZR(G)$ and $g^{-1}h^{-1} \in G_k\setminus G_{k+1}$, we have $g^{-1}=g^{-1}h^{-1}h \succ_1 1$ and this contradicts the fact that $g \sim_{\preceq_1} 1$. This completes the proof.

\end{proof}

\begin{thm} \label{SZR closed}
The set $\SZR(G)$ is closed in $\ZR(G)$ for the Inverse and the Patch topology.
\end{thm}

\begin{proof}
Let $\preceq \in \ZR(G)\setminus\SZR(G)$. Then we know that there exists $g \in G_k\setminus G_{k+1}$ such that $g\succ 1$ and that there exists $h \in G_{k+1}$ such that $gh \preceq 1$. This means that $\preceq \in U:=\U_g \cap \O_{\left(gh\right)^{-1}}$. Then $\preceq$ has a neighborhood composed of non standard preorders. This completes the proof for the Patch topology.\\
For the Inverse topology, we do the same thing hence there exists $g \in G_k\setminus G_{k+1}$ such that $g\succ 1$ and that there exists $h \in G_{k+1}$ such that $gh \preceq 1$. If $gh \prec 1$ then $$\preceq \in U:=\U_g \cap \U_{\left(gh\right)^{-1}}$$ and this completes the proof for the Inverse topology. Otherwise $h$ is equivalent to $g^{-1}$, so $gh^2 \preceq ghg^{-1}=h$ because since $\preceq$ is bi-invariant the subgroup $G_{\preceq}$ is normal in $G$. Hence $gh^2 \prec 1$ because $h\sim_{\preceq} g^{-1}$ and $g \succ 1$. Hence $\preceq \in U:=\U_g \cap \U_{\left(gh^2\right)^{-1}}$ and this completes the proof for the Inverse topology. 
\end{proof}

\begin{rem}
It is not true in general that $\SZR(G)$ is closed in $\ZR(G)$ for the Zariski topology.
\end{rem}

\begin{cor}
The set $\SZR(G)$ is compact for the Inverse and the Patch topology.
\end{cor}

\begin{proof}
This is a consequence of Theorems \ref{compactZR} and \ref{SZR closed}.
\end{proof}

\begin{thm}\label{standardordercantor}
Assume $G \neq \Z$. Then
\begin{enumerate}
\item The set $\SOrd(G)$ is either empty or homeomorphic to the Cantor Set for the Inverse topology.\\
\item The set $\SZR(G)$ is either trivial or homeomorphic to the Cantor Set for the Patch topology.
\end{enumerate}
\end{thm}

\begin{proof}
Let us prove the first asumption.\\
By \cite{S}, it suffices to prove that for any $a_1,\dots,a_n \in G$, the set $$\SOrd(G)\cap \U_{a_1}\cap \dots \cap \U_{a_n}$$ is either empty or infinite.\\
We may assume that all $G_k/G_{k+1}$ are torsion-free. Indeed, if it is not the case, then there exists $k$ such that $\Ord(G_k/G_{k+1})=\emptyset$ hence $\SOrd(G)=\emptyset$ since every order of $\SOrd(G)$ induces an order on $G_k/G_{k+1}$.\\
Therefore for all $k$, $H_k:=G_k/G_{k+1}$ are free abelian groups. It was shown by Baer \cite{Ba} that every torsion-free abelian group $H$ can be embedded in a minimal complete abelian group $H^\ast:=\{(x,n), x\in H, n\in \N^\ast\}$. He defined equality by $$(x,n)=(y,m)\Leftrightarrow mx=ny$$ and addition by $$(x,n)+(y,m)=(mx+ny,mn).$$
Then $H^\ast$ is a complete torsion-free abelian group.\\
We define $f$ to be the application $\ZR(H) \to \ZR(H^\ast)$ given by $\preceq \mapsto \preceq^\ast$ where $\preceq^\ast$ is defined by $$(x,n)\preceq^\ast (y,m) \Leftrightarrow mx\preceq ny.$$
Conversely, we define $g \colon \ZR(H^\ast)\to \ZR(H)$ by $g(\preceq)=\preceq_\ast$ where $\preceq_\ast$ is defined by $$x\preceq_\ast y \Leftrightarrow (x,1)\preceq (y,1)$$
Let $\preceq$ be an element of $\ZR(H)$ and let us show that $g \circ f (\preceq)=\preceq$. We have $g \circ f (\preceq)=g \left(\preceq^\ast\right)=\left(\preceq^\ast\right)_\ast$. To show the equality $\left(\preceq^\ast\right)_\ast=\preceq$, we are going to show that each one refines the other one. First assume $x \preceq y$. Then $(x,1)\preceq^\ast (y,1)$, and this means that $x \left(\preceq^\ast\right)_\ast y$. Now assume $x \left(\preceq^\ast\right)_\ast y$, ie $(x,1) \preceq^\ast (y,1)$ and we do have $x \preceq y$. So $g \circ f =\mathrm{id}_{\ZR(H)}$.\\
Conversely, let $\preceq$ be an element of $\ZR(H^\ast)$ and let us show that $f \circ g(\preceq)=\preceq$. So we want to show the equality $\left(\preceq_\ast\right)^\ast=\preceq$. Take $(x,n)$ and $(y,m)$ in $H^\ast$.
\begin{align*}
(x,n) \left(\preceq_\ast\right)^\ast (y,m) & \Leftrightarrow mx \preceq_\ast ny\\
& \Leftrightarrow (mx,1) \preceq (ny,1)\\
& \Leftrightarrow (mx-ny,1) \preceq (0,1)\\
& \Leftrightarrow mn(mx-ny,mn) \preceq mn(0,1)\\
& \Leftrightarrow (mx-ny,mn) \preceq (0,1) \\
& \Leftrightarrow (x,n) \preceq (y,m).
\end{align*}
Therefore $f \circ g=\mathrm{id}_{\ZR(H\ast)}$, hence there is a bijection between $\ZR(H)$ and $\ZR(H^\ast).$ This bijection induces a bijection between $\Ord(H)$ and $\Ord(H^\ast)$. Furthermore, it is shown in \cite{K} that every complete torsion-free abelian group is isomorphic to $\Q^n$ for some cardinal number $n$. Since $G$ is countable, $H_k$ is countable and $n$ is an integer. Hence $\Ord(H_k)=\Ord(\Q^n)=\Ord(\Z^n)$ for some integer $n$.\\
Then we follow the proof of Proposition 2.3 of \cite{S}, but for the safe of clarity, we include some details.\\
If there exists an integer $k$ such that $\Ord(H_k)=\Ord(\mathbb{Z}^n)$ for some $n\geq2$ then by Proposition 1.7 of \cite{S} we can obtain infinitely many standard orders in $\U_{a_1}\cap \dots \cap \U_{a_n}$. In fact Sikora constructs infinitely many hyperplanes of $\R^n$, each one defining a standard order in $\U_{a_1}\cap \dots \cap \U_{a_n}$.\\
If there exists an integer $k$ such that $\Ord(H_k)=\emptyset$, then we already saw that $\SOrd(G)=\emptyset$.\\
Otherwise, we have $\Ord(H_k)=\Ord(\Z)$ for all $k$, and by example \ref{ZRZ2}, we have $\SOrd(G)=\{0,1\}^{\aleph_0}$ as a topological space. This completes the proof of the first asumption.\\
The proof of the second one is essentially the same. It suffices to prove that for any $a_1,b_1,\dots,a_n,b_n \in G$, the set $$\SZR(G)\cap \U_{a_1}\cap \dots \cap \U_{a_n}\cap \O_{b_1}\cap \dots \cap \O_{b_n}$$ is either trivial or infinite by \cite{DR} (Section 2.11) and Theorem \ref{SZR closed}. All the $H_k$ can be assumed to be torsion-free, otherwise $\SZR(G)$ is trivial. With the same argument we conclude that for all $k$, we have $\ZR(H_k)=\ZR(\Z^n)$ for some integer $n$. Since every order is a preorder, and since the Inverse and the Zariski topology are the same in the space of orders, the case $\ZR(H_k)=\ZR(\Z^n)$ for some $n\geq2$ for some $k$ is exactly the same. Now if there exists $k$ such that $\ZR(H_k)=\{\preceq_{\emptyset}\}$ then $\SZR(G)=\{\preceq_{\emptyset}\}$. So we may assume $\ZR(H_k)=\ZR(\Z)$ for all $k$ and by example \ref{ZRZ2} we have $\SZR(G)=\{0,1,2\}^{\aleph_0}$ as a topological space.
\end{proof}

\begin{rem}
The Proposition \ref{standardordercantor} was shown by Sikora in \cite{S} under the assumption that for every $k$, the quotient $G_k/G_{k+1}$ is finitely generated. He suggested in his paper that it may be possible to relaxe this assumption using the work of \cite{T}, it is what I did.
\end{rem}

\section{Standard valuations}
\begin{defi}\cite{Sc}
A pair $(G,\preceq)$ is called a \emph{simply ordered $\ell$-group} if $G$ is a group, $\preceq\in\Ord(G)$ (in particular it is bi-invariant), and $G$ is lattice with respect to the order $\preceq$.
\end{defi}

\begin{defi}\cite{Sc}\label{def_val}
Let $\K$ be a division ring and $(G,\preceq)$ be a simply ordered $\ell$-group. A \emph{valuation} on $\K$ with values in $G$ is a surjective function $\nu:\K\lgw G\cup\{\infty\}$ such that
\begin{enumerate}
\item[i)] $\nu(0)=\infty\succ u$ for all $u\in G$, and $\nu^{-1}(\infty)=\{0\}$,
\item[ii)] $\nu(uv)=\nu(u)\nu(v)$ for all $u$, $v\in G$,
\item[iii)] $\nu(u+v)\succeq\min\{\nu(u),\nu(v)\}$ for all $u$, $v\in G$.
\end{enumerate}
In this case, the group $G$ is the \emph{value group} of $\nu$ and is denoted by $\G_\nu$.
\end{defi}

\begin{rem}\label{def_val_ring}
The set $V_\nu:=\{x\in\K\mid \nu(x)\geq 1\}$ is a ring with the two following properties:
\begin{enumerate}
\item[a)] $\forall a\in V_\nu$, $\forall b\in\K^*$, $bab^{-1}\in V_\nu$.
\item[b)] $\forall a\in\K^*$ , $a\in V_\nu$ or $a^{-1}\in V_\nu$.
\end{enumerate}
The ring $V_\nu$ is a local ring called the valuation ring of $\nu$ and its maximal ideal is $\m_\nu:=\{x\in\K\mid \nu(x)> 1\}$.
\end{rem}

\begin{lem}[\cite{DR}, Lemma 2.35]\label{subgroup}
Let $G$ be a  group and let $\preceq \in \ZR_l(G)$.
The set 
$$G_\preceq:=\{u\in G\mid u\sim_\preceq 1\}$$
is a subgroup of $G$ called the \emph{residue group} of $\preceq$.\\
Moreover, if $\preceq$ is bi-invariant, then $G_\preceq$ is normal.
\end{lem}
  
Let $G$ be a countable group, $\k$ be a field and $\preceq\in \ZR(G)$ such that $(G/G_\preceq,\preceq)$ is a simply ordered $\ell$-group. In \cite{DR} we saw that every preorder $\preceq$ of a group $G$ defines a monomial valuation $\nu_\preceq$ in the following way:
\begin{align*}
\nu_\preceq \colon  & \K^\k_G \to G/G_\preceq\\
& x^g \mapsto \bar{g}
\end{align*}
where $\K^\k_G:=\{\sum\limits_{g\in G} a_g x^g \mid a_g \in \k \}$, the addition is defined term by term and the multiplication is defined by $x^gx^{g'}=x^{gg'}$.

\begin{defi}
A valuation $\nu_\preceq$ is called \emph{standard} if $\preceq$ is a standard preorder.
\end{defi}

Let $G=G_0 \trianglerighteq G_1 \trianglerighteq \cdots \trianglerighteq G_n \trianglerighteq \cdots$ be the lower central series of a group $G$ such that $\bigcap\limits_{k=0}^\infty G_k=\{1\}$.
\begin{prop}\label{val standard max}
Let $\nu_\preceq$ be a standard valuation, $h_0\in G_0 \setminus G_1$ such that $h_0>1$ and $P\in  \K^\k_G$. Then
\begin{enumerate}
\item If $\nu_\preceq(P)=\overline{g_0}$ where $g_0\in G_0 \setminus G_1$ such that $g_0 \preceq h_0^{-1}$, then $x^{g_0^{-1}h_0}P(x) \in \mathfrak{m}_{\nu_\preceq}$.
\item Otherwise $x^{h_0}P(x) \in \mathfrak{m}_{\nu_\preceq}$
\end{enumerate}
\end{prop}

\begin{proof}
In the first case, we have $g_0^{-1}h_0g_0=h_0h_0^{-1}g_0^{-1}h_0g_0$ but $h_0\in G_0 \setminus G_1$, $h_0>1$ and $h_0^{-1}g_0^{-1}h_0g_0 \in G_1$ because $(G_i)$ is the lower central series of $G$. Since $\preceq$ is a standard preorder, we have $g_0^{-1}h_0g_0 >1$, then $\nu_\preceq\left(x^{g_0^{-1}h_0}P(x)\right)=\overline{g_0^{-1}h_0g_0} >1$.\\
If $\nu_\preceq(P)=\overline{g_0}$ where $g_0\in G_0 \setminus G_1$, $g_0 <1$ and $h_0^{-1}<g_0$, then $\nu_\preceq\left(x^{h_0}P(x)\right)=\overline{h_0g_0}>1$.\\
If $\nu_\preceq(P)=\overline{g_0}$ where $g_0>1$, then $\nu_\preceq\left(x^{h_0}P(x)\right)=\overline{h_0g_0}>1$.\\
The last case is when $\nu_\preceq(P)=\overline{g_0}$ where $g_0 <1$ and $g_0 \in G_n \setminus G_{n+1}$ with $n \geq 1$. Let $$\left(h_1,\cdots,h_{n-1}\right) \in \left(G_1 \setminus G_2 \right)\times \cdots \times \left( G_{n-1} \setminus G_n\right)$$ such that $h_i >1$ for all $i$. Then 
$$h_0g_0=(h_0h_1^{-1})(h_1h_2^{-1})(h_2h_3^{-1})\cdots (h_{n-2}h_{n-1}^{-1})(h_{n-1}g_0)>1.$$
Because $\preceq$ is standard, $h_i \in G_i \setminus G_{i+1}$, $h_i>1$ for all $i$. This completes the proof.
\end{proof}

\begin{rem}
Since the composition of preorders corresponds to the composition of valuations (see \cite{DR} Proposition 4.8) and by Proposition \ref{standard compo}, the conclusion of Propositon \ref{val standard max} holds for the composition of two standard valuations.
\end{rem}

\FloatBarrier

\end{document}